\documentclass{amsart}

\usepackage{graphicx}
\usepackage[centertags]{amsmath}
\usepackage{amsfonts}
\usepackage{amssymb}
\usepackage{amsthm}
\usepackage{newlfont}
\usepackage{hyperref}

\newtheorem{thm}{Theorem}[section]

\newtheorem{lem}[thm]{Lemma}
\newtheorem{prop}[thm]{Proposition}

\theoremstyle{definition}

\theoremstyle{remark}

\numberwithin{equation}{section}

\newcommand{\N}{\mathbb{N}}
\newcommand{\Z}{\mathbb{Z}}
\newcommand{\R}{\mathbb{R}}

\newcommand{\acts}{\curvearrowright}
\newcommand{\symd}{\triangle}

\newcommand{\Sym}{\mathrm{Sym}}

\newcommand{\sH}{\mathrm{H}}

\newcommand{\juxt}{\frown}
\renewcommand{\:}{\,:\,}
\newcommand{\res}{\restriction}

\begin{document}

\title[Ergodic actions and finite generators]{Ergodic actions of countable groups and finite generating partitions}


\author{Brandon Seward}
\address{Department of Mathematics, University of Michigan, 530 Church Street, Ann Arbor, MI 48109, U.S.A.}
\email{b.m.seward@gmail.com}
\keywords{finite generator, generating partition, Shannon entropy, Krieger's finite generator theorem, ergodic, countable groups, f-invariant, sofic}

\begin{abstract}
We prove that if an ergodic action of a countable group on a probability space admits a generating partition having finite Shannon entropy then it admits a finite generating partition.
\end{abstract}
\maketitle

\section{Introduction}

Let $G$ be a countable group acting by measure preserving bijections on a probability space $(X, \mu)$. For a partition $\alpha$ of $X$, we denote by $G \cdot \alpha$ the smallest $\sigma$-algebra containing the sets $g \cdot A$ for $g \in G$ and $A \in \alpha$. The partition $\alpha$ is \emph{generating} (or a \emph{generator}) if for every measurable set $B \subseteq X$ there is some $B' \in G \cdot \alpha$ with $\mu(B \symd B') = 0$. The \emph{Shannon entropy} of a partition $\alpha$ is
$$\sH(\alpha) = \sum_{A \in \alpha^*} - \mu(A) \cdot \log(\mu(A)).$$
if there is a countable subcollection $\alpha^* \subseteq \alpha$ with $\mu(\cup \alpha^*) = 1$, and otherwise $\sH(\alpha) = \infty$.

A classical theorem of ergodic theory is Krieger's finite generator theorem \cite{Kr70}. This theorem states that if $k$ is an integer and $\Z \acts (X, \mu)$ is an ergodic action with Kolmogorov--Sinai entropy less than $\log(k)$, then the action $\Z \acts (X, \mu)$ admits a finite generating partition of size $k$. Less known is a similar but earlier result of Rohlin \cite{Roh67}. Rohlin proved that for any essentially free action $\Z \acts (X, \mu)$, the Kolmogorov--Sinai entropy of this action is equal to the infimum of the Shannon entropies of the generating partitions. A result stronger than both Rohlin's theorem and Krieger's theorem was obtained by Denker in \cite{D74}.

Krieger's finite generator theorem extends to essentially free ergodic actions of general countable amenable groups. This more general version of Krieger's theorem was stated by \v{S}ujan in \cite{Su83} and proved by Danilenko and Park in \cite{DP02} with the more restrictive requirement that the Kolmogorov--Sinai entropy be less than $\log(k-1)$ (Rosenthal \cite{Ros88} proved this using $\log(k-2)$ and Thouvenot \cite{Th75} proved this for actions of $\Z^n$ using $\log(k-2)$). The extension of Rohlin's theorem to essentially free actions of amenable groups is not in the literature, but it follows from the techniques of Danilenko and Park in \cite{DP02}. So in the setting of actions of amenable groups, the behavior of generating partitions is well understood.

Entropy theory recently has been extended beyond the realm of actions of amenable groups. In 2008, Lewis Bowen defined (f-invariant) entropy for actions of finitely generated free groups \cite{B10a} and (sofic) entropy for actions of sofic groups \cite{B10b}. The definition of sofic entropy was later expanded by Kerr--Li \cite{KL11a} (see also \cite{Ke}). Sofic entropy and f-invariant entropy have strong similarities with Kolmogorov--Sinai entropy, and in fact when the acting group is amenable these entropies agree with Kolmogorov--Sinai entropy. Furthermore, f-invariant entropy is essentially a special case of sofic entropy. The theories of sofic entropy and f-invariant entropy, being quite new, are currently poorly understood. In particular, it is not clear what relationships sofic entropy and f-invariant entropy have with generating partitions. As with Kolmogorov--Sinai entropy, sofic entropy and f-invariant entropy are easier to use, define, and compute when there are generating partitions with finite Shannon entropy (in fact, the definition of f-invariant entropy still requires a generating partition having finite Shannon entropy). So theorems along the lines of Krieger's theorem and Rohlin's theorem mentioned above would certainly benefit the theories of sofic entropy and f-invariant entropy. Thus the question arises as to what can be said about generating partitions outside of the realm of actions of amenable groups.

Although we draw motivation from sofic entropy and f-invariant entropy (which deal with actions of sofic groups and finitely generated free groups, respectively), our main theorem deals with actions of general countable groups. We prove the following finite generator theorem.

\begin{thm} \label{INTRO THM}
Let $G$ be a countable group acting ergodically by measure preserving bijections on a standard probability space $(X, \mu)$. If this action admits a generating partition having finite Shannon entropy, then it admits a finite generating partition.
\end{thm}

We mention that although our proof is constructive, it does not immediately imply any relationship between the Shannon entropy of the original partition and the size of the finite partition constructed.

We obtained the above theorem while trying to establish a Krieger finite generator theorem for f-invariant entropy. However, we found that Krieger's theorem and Rohlin's theorem fail in this setting.

For an action $G \acts (X, \mu)$ we let $\Delta^*_G(X, \mu)$ denote the infimum of the Shannon entropies of the generating partitions (this is $+\infty$ if there are no such partitions) and we let $\Delta_G(X, \mu)$ denote the smallest size of a finite generating partition (again this is $+\infty$ if there are no such partitions). If $G$ is a finitely generated free group then we denote the f-invariant entropy of the action $G \acts (X, \mu)$ by $f_G(X, \mu)$. We remark that $f_G(X, \mu)$ is only defined when there is a generating partition having finite Shannon entropy. When $f_G(X, \mu)$ is defined it takes values in $\R \cup \{-\infty\}$.

\begin{prop} \label{INTRO PROPF}
Let $G$ be a finitely generated non-cyclic free group. For every $h \in \R$,
$$\sup_{G \acts (X, \mu)} \Delta^*_G(X, \mu) = \sup_{G \acts (X, \mu)} \Delta_G(X, \mu) = + \infty,$$
where the supremums are taken over all essentially free ergodic probability measure preserving actions $G \acts (X, \mu)$ with $f_G(X, \mu)$ defined and $f_G(X, \mu) = h$.
\end{prop}

Since f-invariant entropy is currently only defined when there is a generating partition having finite Shannon entropy, it follows from our main theorem that $\Delta^*_G(X, \mu)$ and $\Delta_G(X, \mu)$ are finite for any ergodic action $G \acts (X, \mu)$ in which $f_G(X, \mu)$ is defined (so in particular for the actions considered in the proposition above).

We wish to emphasize that this proposition only says that the \emph{most obvious} translations of Krieger's theorem and Rohlin's theorem to the setting of f-invariant entropy are false. There is still opportunity for subtle modifications of Krieger's theorem and Rohlin's theorem to be true for f-invariant entropy (we will mention such a possible modification).

We obtain a similar but much weaker result for sofic entropy. With sofic entropy the situation is much different though because the sofic entropy of a sofic group action $G \acts (X, \mu)$ is not a single number but a collection of numbers. These numbers are indexed by sofic approximation sequences to the group. It is known that different sofic approximation sequences can give different sofic entropy values, however it is not yet understood how widespread this phenomena is. For a sofic group $G$, a sofic approximation sequence $\Sigma$ to $G$, and an action $G \acts (X, \mu)$, we denote the corresponding sofic entropy by $h^\Sigma_G(X, \mu)$. We remark that $h^\Sigma_G(X, \mu)$ is always defined and takes values in $\{-\infty\} \cup [0, +\infty]$.

\begin{prop} \label{INTRO PROPS}
There exists a sofic group $G$, a sofic approximation sequence $\Sigma$ to $G$, and an essentially free ergodic action $G \acts (X, \mu)$ such that $h^\Sigma_G(X, \mu) = -\infty$ but $\Delta^*_G(X, \mu) = \Delta_G(X, \mu) = +\infty$.
\end{prop}

We do not view the above corollary as sufficient grounds to say that Krieger's theorem and Rohlin's theorem fail for sofic entropy. The fact is that having $h^\Sigma_G(X, \mu) = -\infty$ reflects almost nothing about the action; it only means that $\Sigma$ is inadequate for modeling the action $G \acts (X, \mu)$. We do not know if Rohlin's theorem and Krieger's theorem hold for actions whose sofic entropy is not negative infinity.

\subsection*{Organization}
In Section \ref{SEC PROOF} below, we prove Theorem \ref{INTRO THM}. Then in Section \ref{SEC COUNTER} we prove Propositions \ref{INTRO PROPF} and \ref{INTRO PROPS}. These two sections are written independently of one another.

\subsection*{Acknowledgments}
This material is based upon work supported by the National Science Foundation Graduate Student Research Fellowship under Grant No. DGE 0718128. The author would like to thank his advisor, Ralf Spatzier, for helpful conversations. The author would also like to thank Benjamin Weiss for some references and for information on Rohlin's theorem.

\section{Construction of finite generators} \label{SEC PROOF}

We begin with an equivalent characterization of generating partitions. For a partition $\alpha$ of $X$ and a point $x \in X$, we write $\alpha(x)$ to denote the unique member $A \in \alpha$ with $x \in A$.

\begin{lem} \label{LEM GENEQ}
Let $X$ be a standard Borel space, let $\mu$ be a Borel probability measure on $X$, and let $G$ be a countable group acting by measure preserving bijections on $X$. The following are equivalent for a countable measurable partition $\alpha$ of $X$
\begin{enumerate}
\item[\rm (i)] $\alpha$ is a generating partition;
\item[\rm (ii)] for every Borel set $B \subseteq X$ there is a set $B' \in G \cdot \alpha$ with $\mu(B \symd B') = 0$;
\item[\rm (iii)] there is $X' \subseteq X$ such that $\mu(X') = 1$ and for all $x \neq y \in X'$ there is $g \in G$ with $\alpha(g \cdot x) \neq \alpha(g \cdot y)$.
\end{enumerate}
\end{lem}

\begin{proof}
The equivalence of (i) and (ii) is by definition.

(ii) $\Rightarrow$ (iii). Since $X$ is a standard Borel space, there is a countable collection $(B_n)_{n \in \N}$ of Borel subsets of $X$ such that the $\sigma$-algebra generated by $\{B_n \: n \in \N\}$ is precisely the collection of all Borel subsets of $X$ and such that $(B_n)_{n \in \N}$ separates points, meaning that for $x \neq y \in X$ there is $n \in \N$ with either $x \in B_n$ and $y \not\in B_n$ or $x \not\in B_n$ and $y \in B_n$ \cite[Proposition 12.1]{K95}. For each $n \in \N$ let $C_n \in G \cdot \alpha$ be such that $\mu(B_n \symd C_n) = 0$. Set
$$X' = X \setminus \bigcup_{n \in \N} (B_n \symd C_n).$$
Then $\mu(X') = 1$. Now fix $x, y \in X'$ with $x \neq y$. Then there is $n \in \N$ with $B_n$ containing either $x$ or $y$ but not containing both. Since $x, y \in X'$ we have $x \in B_n$ if and only if $x \in C_n$, and similarly $y \in B_n$ if and only if $y \in C_n$. Therefore $C_n$ contains either $x$ or $y$, but it does not contain both. Since $C_n$ lies in the $\sigma$-algebra generated by the sets $\{g \cdot A \:g \in G, \ A \in \alpha\}$, there must be $g \in G$ and $A \in \alpha$ with $g \cdot A$ containing either $x$ or $y$ but not both. Then $\alpha(g^{-1} \cdot x) \neq \alpha(g^{-1} \cdot y)$.

(iii) $\Rightarrow$ (ii). Let $X' \subseteq X$ be such that $\mu(X') = 1$ and for all $x \neq y \in X'$ there is $g \in G$ with $\alpha(g \cdot x) \neq \alpha(g \cdot y)$. By replacing $X'$ with $\bigcap_{g \in G} g \cdot X'$ if necessary, we may suppose that $g \cdot X' = X'$ for every $g \in G$. Let $\beta$ be the partition of $X'$ induced by $\alpha$. Since $\mu(X \setminus X') = 0$, it suffices to show that for every Borel set $B \subseteq X$ we have $B \cap X' \in G \cdot \beta$. Consider the space $\beta^G$ of all functions from $G$ to $\beta$. We give $\beta^G$ the topology of point-wise convergence (under the discrete topology on $\beta$). Then $\beta^G$ is a Polish space since $\beta$ is countable. We let $G$ act on $\beta^G$ by the rule $(h \cdot z)(g) = z(h^{-1} g)$ for $h, g \in G$ and $z \in \beta^G$. Define $\phi: X' \rightarrow \beta^G$ by $\phi(x)(g) = \beta(g^{-1} \cdot x)$. Notice that
$$\phi(h \cdot x)(g) = \beta(g^{-1} h \cdot x) = \phi(x)(h^{-1} g) = [h \cdot \phi(x)](g).$$
So $\phi(h \cdot x) = h \cdot \phi(x)$. If $x \neq y \in X'$ then by assumption there is $g \in G$ with $\beta(g \cdot x) \neq \beta(g \cdot y)$ and thus $\phi(x) \neq \phi(y)$. So $\phi$ is injective. The function $\phi$ is Borel since the inverse image of any open set in $\beta^G$ is Borel. Consider the partition $\xi = \{C_A \: A \in \beta\}$ of $\beta^G$, where $C_A = \{z \in \beta^G \: z(1_G) = A\}$. It is readily seen that the $G$-translates of the members of $\xi$ generate the open subsets of $\beta^G$. Therefore $G \cdot \xi$ is precisely the collection of Borel subsets of $\beta^G$. Notice that $A \subseteq \phi^{-1}(C_A)$ for $A \in \beta$ and thus $A = \phi^{-1}(C_A)$ for $A \in \beta$ (since $\phi^{-1}(\xi)$ and $\beta$ are both partitions of $X'$). Consider the collection $\mathcal{C}$ of subsets $C \subseteq \beta^G$ such that $\phi^{-1}(C) \in G \cdot \beta$. Clearly $\mathcal{C}$ is a $\sigma$-algebra and $g \cdot C_A \in \mathcal{C}$ for every $g \in G$ and $C_A \in \xi$. Thus $\phi^{-1}(C) \in G \cdot \beta$ for every Borel set $C \subseteq \beta^G$. Now consider a Borel set $B \subseteq X$. Since $\phi$ is injective and Borel, $\phi(B \cap X')$ is a Borel subset of $\beta^G$ \cite[Corollary 15.2]{K95}. Therefore $B \cap X' = \phi^{-1}(\phi(B \cap X')) \in G \cdot \beta$. So there is $B' \in G \cdot \alpha$ with $B' \cap X' = B \cap X'$ and thus $\mu(B \symd B') = 0$ since $\mu(X \setminus X') = 0$.
\end{proof}

For a finite set $S$ we let $S^{< \omega}$ denote the set of all finite words with letters in $S$ (the $\omega$ in the superscript denotes the first infinite ordinal). For $z \in S^{< \omega}$ we let $|z|$ denote the length of the word $z$. In the first step of the proof of \cite[Theorem 2.1]{Kr70}, Krieger proves the following. 

\begin{lem}[Krieger] \label{LEM KRIEGER}
Let $(X, \mu)$ be a probability space. If $\alpha$ is a countable measurable partition of $X$ with $\sH(\alpha) < \infty$ then there exists an injection $L : \alpha \rightarrow \{1, 2, 3\}^{<\omega}$ such that
$$\sum_{A \in \alpha} |L(A)| \cdot \mu(A) < \infty.$$
\end{lem}

As a convenience to the reader, we include the proof below.

\begin{proof}
This is clear if $\alpha$ is finite. So suppose that $\alpha$ is countably infinite and enumerate $\alpha$ as $\alpha = \{A_1, A_2, \ldots\}$, where $\mu(A_{m+1}) \leq \mu(A_m)$ for all $m$. For $m \geq 1$ choose $t(m) \in \N$ so that $-\log(\mu(A_m)) - 1 < t(m) \leq -\log(\mu(A_m))$. Then
$$3^{-t(m)} \leq e^{-t(m)} \leq e^{\log(\mu(A_m)) + 1} = e \cdot \mu(A_m).$$
Thus $\sum_{m = 1}^\infty 3^{-t(m)} \leq e$. Notice that the sequence $(t(m))_{m = 1}^\infty$ is non-decreasing. Set $N_1 = 1$ and for $m > 1$ define
$$N_m = \min\{k > N_{m-1} \: t(k) > t(N_{m-1})\}.$$
Then
$$\sum_{m = 1}^\infty (N_{m+1} - N_m) \cdot 3^{-t(N_m)} = \sum_{m=1}^\infty 3^{-t(m)} \leq e.$$
So there is $m_0 \geq 1$ such that $N_{m+1} - N_m < 3^{t(N_m)}$ for all $m \geq m_0$. Therefore it is possible to choose $L(A_m) \in \{1, 2, 3\}^{t(m)}$ for every $m \geq m_0$ so that $L : \{A_{m_0}, A_{m_0+1}, \ldots\} \rightarrow \{1, 2, 3\}^{< \omega}$ is injective. Moreover, since the inequality $N_{m+1} - N_m < 3^{t(N_m)}$ (which holds for $m \geq m_0$) is strict, $L$ can be extended to an injective function $L : \alpha \rightarrow \{1, 2, 3\}^{< \omega}$. Finally, it suffices to show that
$$\sum_{m = m_0}^\infty |L(A_m)| \cdot \mu(A) < \infty.$$
This follows from the fact that $|L(A_m)| = t(m) \leq - \log(\mu(A_m))$ for $m \geq m_0$ and $\sum_{m = m_0}^\infty -\log(\mu(A_m)) \cdot \mu(A_m) \leq \sH(\alpha) < \infty$. 
\end{proof}

The function $L$ above can be extended to $X$ by setting $L(x) = L(\alpha(x))$. The above lemma then says that the labeling $L : X \rightarrow \{1, 2, 3\}^{< \omega}$ has finite length on average. The idea behind the proof of Theorem \ref{INTRO THM} is to rearrange the L ``data'' within each orbit to obtain a new function (a relabeling) $R: X \rightarrow \{1, 2, 3, 4\}^{< \omega}$ which has uniformly bounded length. The function $R$ would then have finite image and thus induce a finite partition of $X$. In order for this partition to be generating, one must ensure that the function $L$ can be recovered from $R$. In \cite{Kr70}, Krieger carried out this argument in the case of $\Z$ actions, obtaining a weak form of his finite generator theorem which did not specify the smallest possible size of a finite generator. While our proof is inspired by his argument, our proof is quite distinct as Krieger's argument relied heavily on properties of $\Z$.

The following lemma is essential for the task of rearranging the $L$ data within each orbit.

If $G$ acts on $(X, \mu)$ and $A \subseteq X$, then we say that $x, y \in X$ are \emph{$A$-symmetric} if for every $g \in G$ $g \cdot x \in A \Leftrightarrow g \cdot y \in A$.

\begin{lem} \label{LEM BIJECT}
Let $G$ be a countable group acting ergodically by measure preserving bijections on a probability space $(X, \mu)$. For every pair of measurable sets $A, B \subseteq X$ there exist measurable sets $P_1(A, B), P_2(A, B) \subseteq X$ and a measurable bijection $\psi(A, B) : P_1(A, B) \rightarrow P_2(A, B)$ satisfying the following:
\begin{enumerate}
\item[\rm (i)] $P_1(A, B) \subseteq A$ and $P_2(A, B) \subseteq B$;
\item[\rm (ii)] either $\mu(A \setminus P_1(A, B)) = 0$ or $\mu(B \setminus P_2(A, B)) = 0$;
\item[\rm (iii)] if $x, y \in X$ are both $A$-symmetric and $B$-symmetric, then $x$ and $y$ are both $P_1(A, B)$-symmetric and $P_2(A, B)$-symmetric;
\item[\rm (iv)] $\psi(A, B)(x) \in G \cdot x$ for every $x \in P_1(A, B)$;
\item[\rm (v)] if $x, y \in P_1(A, B)$ are both $A$-symmetric and $B$-symmetric, then there is $h \in G$ with $\psi(A, B)(x) = h \cdot x$ and $\psi(A, B)(y) = h \cdot y$.
\end{enumerate}
\end{lem} 

\begin{proof}
We first define auxiliary functions $Q_1$, $Q_2$, and $\theta$ which will play roles similar to $P_1$, $P_2$, and $\psi$, respectively. The idea is to define $Q_1$, $Q_2$, and $\theta$ to achieve clauses (i), (iii), (iv), and (v) and then use these functions repeatedly to perform a type of exhaustion argument and achieve clause (ii).

Fix any well ordering, $\preceq$, of $G$. If $A, B \subseteq X$ are measurable and $(G \cdot A) \cap B = \varnothing$ then we set $Q_1(A, B) = Q_2(A, B) = \theta(A, B) = \varnothing$. If $(G \cdot A) \cap B \neq \varnothing$ then we let $h \in G$ be the $\preceq$-least element of $G$ satisfying $(h \cdot A) \cap B \neq \varnothing$. In this case we set
$$Q_1(A, B) = A \cap h^{-1} \cdot B$$
$$Q_2(A, B) = (h \cdot A) \cap B$$
$$\theta(A, B)(x) = h \cdot x \ (\text{for } x \in Q_1(A, B)).$$
Notice that $Q_1(A, B)$ and $Q_2(A, B)$ are measurable subsets of $X$ and $\theta(A, B)$ is a measurable function. If $P_1$, $P_2$, and $\psi$ are replaced with $Q_1$, $Q_2$, and $\theta$, respectively, then clauses (i), (iv), and (v) are clearly satisfied (clause (v) immediately follows from the simple definition of $\theta(A, B)$). Clause (iii) is also satisfied, for if we assume $Q_1(A, B)$ and $Q_2(A, B)$ are non-empty (clause (iii) is trivial otherwise) and let $h$ be as above, then for any $x \in X$ and $g \in G$ we have $g \cdot x \in Q_1(A, B) \Leftrightarrow (g \cdot x \in A \wedge h g \cdot x \in B)$ and $g \cdot x \in Q_2(A, B) \Leftrightarrow (g \cdot x \in B \wedge h^{-1} g \cdot x \in A)$.

We have $Q_1(A, B) \subseteq A$ and $Q_2(A, B) \subseteq B$. By repeatedly using the functions $Q_1$ and $Q_2$ we seek to exhaust (in measure) either $A$ or $B$. We recursively define
$$P_i^1(A, B) = Q_i(A, B),$$
$$P_i^n(A, B) = Q_i(A \setminus P_1^{n-1}(A, B), B \setminus P_2^{n-1}(A, B)) \cup P_i^{n-1}(A, B),$$
$$\psi^n(A, B) = \theta(A \setminus P_1^{n-1}(A, B), B \setminus P_2^{n-1}(A, B)).$$
We set $P_i(A, B) = \bigcup_{n \geq 1} P_i^n(A, B)$ and $\psi(A, B) = \bigcup_{n \geq 1} \psi^n(A, B)$. Clearly $P_i(A, B)$ is a measurable subset of $X$ and $\psi(A, B)$ is a measurable function. We remark that $\psi(A, B)$ is a well defined function since the domains of the $\psi^n$'s are pairwise disjoint.

Clauses (i) and (iv) are clearly satisfied. We now check clauses (ii), (iii), and (v).

(ii). Let $h_1$ be the $\preceq$-least element of $G$ with $(h_1 \cdot A) \cap B \neq \varnothing$ and for $n > 1$ let $h_n$ be the $\preceq$-least element of $G$ with
$$(h_n \cdot (A \setminus P_1^{n-1}(A, B))) \cap (B \setminus P_2^{n-1}(A, B)) \neq \varnothing.$$
If for some $n$ no such $h_n$ exists then by ergodicity either $\mu(A \setminus P_1^{n-1}(A, B)) = 0$ or $\mu(B \setminus P_2^{n-1}(A, B)) = 0$ and thus clause (ii) is satisfied since $P_i^{n-1}(A, B) \subseteq P_i(A, B)$. So we may suppose the $h_n$'s are defined. We must have $h_n \prec h_{n+1}$ since $P_i^{n-1}(A, B) \subseteq P_i^n(A, B)$. So if $g \in G$ is fixed then
$$(g \cdot (A \setminus P_1^{n-1}(A, B))) \cap (B \setminus P_2^{n-1}(A, B)) = \varnothing$$
for all but finitely many $n \geq 1$. Thus
$$(g \cdot (A \setminus P_1(A, B))) \cap (B \setminus P_2(A, B)) = \varnothing$$
for every $g \in G$. By ergodicity it follows that either $\mu(A \setminus P_1(A, B)) = 0$ or $\mu(B \setminus P_2(A, B)) = 0$.

(iii). Fix $x, y \in X$ which are both $A$-symmetric and $B$-symmetric. Then for $i = 1, 2$ we have that $x$ and $y$ are $P_i^1(A, B)$-symmetric, since $P_i^1(A, B) = Q_i(A, B)$. Now suppose that $x$ and $y$ are $P_i^{n-1}(A, B)$-symmetric for $i = 1, 2$. Then $x$ and $y$ are $(A \setminus P_1^{n-1}(A, B))$-symmetric and $(B \setminus P_2^{n-1}(A, B))$-symmetric. It follows from the definition of $P_i^n$ and the properties of $Q_i$ that $x$ and $y$ are $P_i^n(A, B)$ symmetric for $i = 1, 2$. By induction, this holds for all $n \geq 1$. Thus $x$ and $y$ are $P_i(A, B)$-symmetric for $i = 1, 2$.

(v). Fix $x, y \in P_1(A, B)$ which are both $A$-symmetric and $B$-symmetric. Let $n \geq 1$ be such that $x$ lies in the domain of $\psi^n(A, B)$. Notice that the domain of $\psi^n(A, B)$ is $P_1^n(A, B) \setminus P_1^{n-1}(A, B)$. The argument in the previous paragraph shows that $x$ and $y$ are $P_1^k(A, B)$-symmetric for every $k \geq 1$. Therefore $y$ lies in the domain of $\psi^n(A, B)$ as well. So
$$x, y \in P_1^n(A, B) \setminus P_1^{n-1}(A, B) = Q_1(A \setminus P_1^{n-1}(A, B), B \setminus P_2^{n-1}(A, B))$$
and $x$ and $y$ are both $(A \setminus P_1^{n-1}(A, B))$-symmetric and $(B \setminus P_2^{n-1}(A, B))$-symmetric, so from the properties of $Q_1$, $Q_2$, and $\theta$, it follows that there is $h \in G$ with
$$\psi(A, B)(x) = \psi^n(A, B)(x) = \theta(A \setminus P_1^{n-1}(A, B), B \setminus P_2^{n-1}(A, B))(x) = h \cdot x$$
and
$$\psi(A, B)(y) = \psi^n(A, B)(y) = \theta(A \setminus P_1^{n-1}(A, B), B \setminus P_2^{n-1}(A, B))(y) = h \cdot y.$$
This completes the proof.
\end{proof}

We are now ready to prove the main theorem.

\begin{thm}
Let $G$ be a countable group acting ergodically by measure preserving bijections on a standard probability space $(X, \mu)$. If this action admits a generating partition having finite Shannon entropy, then it admits a finite generating partition.
\end{thm}

\begin{proof}
Let $\alpha$ be a generating partition with $\sH(\alpha) < \infty$. By combining the classes of $\alpha$ having measure $0$ into a single class, we may suppose that $\alpha$ is countable. By Lemma \ref{LEM KRIEGER}, there is an injective function $L : \alpha \rightarrow \{1, 2, 3\}^{< \omega}$ satisfying
$$\sum_{A \in \alpha} |L(A)| \cdot \mu(A) < \infty.$$
For $x \in X$ define $L(x) = L(\alpha(x))$.

The function $L: X \rightarrow \{1, 2, 3\}^{< \omega}$ on average has finite length, so the idea now is to rearrange the $L$-data within each orbit so that in the end every point of $X$ has a word of uniformly finite length associated to it. In doing this, one must take care not to lose data, and more importantly one must rearrange the data in such a way that the original function $L$ can be decoded from the new data. The functions appearing in the previous lemma play the critical role of achieving these requirements.

First we need to determine how long the new words we create should be. Since $|L(x)|$ is integer valued, we have
$$\sum_{n \geq 1} n \cdot \mu(\{x \in X \: |L(x)| = n\}) = \sum_{A \in \alpha} |L(A)| \cdot \mu(A) < \infty.$$
So there is $C \geq 1$ such that
$$\sum_{n > C} n \cdot \mu(\{x \in X \: |L(x)| = n\}) < \frac{1}{4}.$$
After rearranging the $L$-data, the new words we construct will have length bounded above by $C + 2$.

We now use the previous lemma to determine how to rearrange the $L$-data. For $n \geq 1$ define $B_n = \{x \in X \: |L(x)| \geq C + n\}$. Then $B_{n+1} \subseteq B_n$. Since $|L(x)| - C = |\{n \geq 1 \: x \in B_n\}|$ when $|L(x)| > C$, we have
$$\mu(B_1) \leq \sum_{n \geq 1} \mu(B_n) = \sum_{n \geq 1} n \cdot \mu(\{x \in X \: |L(x)| = C + n\}) < \frac{1}{4}.$$
Set $T_1 = P_2(B_1, X \setminus B_1)$ and $\theta_1 = \psi(B_1, X \setminus B_1)$. In general, for $n > 1$ define
$$T_n = P_2(B_n, X \setminus (B_1 \cup T_1 \cup \cdots \cup T_{n-1})),$$
$$\theta_n = \psi(B_n, X \setminus (B_1 \cup T_1 \cup \cdots \cup T_{n-1})).$$
By clause (iv) of Lemma \ref{LEM BIJECT}, the function $\theta_n$ can be described by partitioning $P_1(B_n, X \setminus (B_1 \cup T_1 \cup \cdots \cup T_{n-1}))$ into a countable number of pieces and translating each piece by an element of $G$. Since $\theta_n$ is bijective, it follows that $\mu(T_n) = \mu(\theta_n^{-1}(T_n))$. Therefore by clause (i) of Lemma \ref{LEM BIJECT}
$$\mu(T_n) = \mu(\theta_n^{-1}(T_n)) = \mu(P_1(B_n, X \setminus (B_1 \cup T_1 \cup \cdots \cup T_{n-1}))) \leq \mu(B_n).$$
So
$$\mu(B_1 \cup T_1 \cup \cdots \cup T_{n-1}) \leq \mu(B_1) + \mu(B_1) + \mu(B_2) + \cdots + \mu(B_{n-1}) < \frac{1}{2}$$
and hence $\mu(B_n) < \frac{1}{2} < \mu(X \setminus (B_1 \cup T_1 \cup \cdots \cup T_{n-1}))$. Applying clause (ii) of Lemma \ref{LEM BIJECT}, we find that
$$\mu(B_n \setminus P_1(B_n, X \setminus (B_1 \cup T_1 \cup \cdots \cup T_{n-1}))) = 0.$$
Set $E_n = B_n \setminus P_1(B_n, X \setminus (B_1 \cup T_1 \cup \cdots \cup T_{n-1}))$.

We now define a new labeling function $R : X \rightarrow \{1, 2, 3, 4\}^{< \omega}$ by the rule (below the symbol $\juxt$ denotes concatenation of words and $\res$ denotes restriction)
$$R(x) = \begin{cases}
L(x) \res [1, C] & \text{if } x \in B_1 \\
L(x) \juxt 4 \juxt L(\theta_n^{-1}(x))(C + n) & \text{if } x \in T_n \text{ and } \theta_n^{-1}(x) \in B_{n+1}\\
L(x) \juxt 4 \juxt L(\theta_n^{-1}(x))(C + n) \juxt 4 & \text{if } x \in T_n \text{ and } \theta_n^{-1}(x) \not\in B_{n+1}\\
L(x) \juxt 4 & \text{otherwise}.
\end{cases}$$
Notice that in the second and third cases in the definition of $R(x)$ we automatically have $\theta_n^{-1}(x) \in B_n$ since $x \in T_n$. Clearly $|R(x)| < C+3$ for every $x \in X$. So the image of $R$ is finite. Let $\beta$ be the partition of $X$ obtained from $R$, i.e. define the classes of $\beta$ so that $x, y \in X$ lie in the same class of $\beta$ if and only if $R(x) = R(y)$. Then $\beta$ is a finite measurable partition of $X$. We claim that $\beta$ is a generating partition.

By Lemma \ref{LEM GENEQ} and the definition of $\beta$, we have that $\beta$ is a generating partition if and only if there is a set $X' \subseteq X$ with $\mu(X') = 1$ such that for every $x \neq y \in X'$ there is $g \in G$ with $R(g \cdot x) \neq R(g \cdot y)$. Since $\alpha$ is a countable generating partition, there is a set $X'' \subseteq X$ with $\mu(X'') = 1$ such that for all $x \neq y \in X''$ there is $g \in G$ with $\alpha(g \cdot x) \neq \alpha(g \cdot y)$. Set
$$X' = X'' \setminus \left( G \cdot \bigcup_{n \geq 1} E_n \right).$$
Then $\mu(X') = 1$. Fix $x \neq y \in X'$. We proceed by cases to show that there is $g \in G$ with $R(g \cdot x) \neq R(g \cdot y)$.

\underline{Case 1}: There is $n \geq 1$ such that $x$ and $y$ are not $B_n$-symmetric. Let $n \geq 1$ be least such that there is $u \in G$ with $B_n$ containing precisely one of $u \cdot x$ and $u \cdot y$. To be specific, say $u \cdot x \in B_n$ and $u \cdot y \not\in B_n$ (the other case is nearly identical). If $n = 1$ then the letter $4$ appears in $R(u \cdot y)$ but not $R(u \cdot x)$ and thus $R(u \cdot x) \neq R(u \cdot y)$. So suppose that $n > 1$. Since $n$ was chosen to be minimal, we must have that $x$ and $y$ are $B_k$-symmetric for all $k < n$. So by clause (iii) of Lemma \ref{LEM BIJECT} $x$ and $y$ are $T_1$-symmetric. It readily follows from a simple induction argument that $x$ and $y$ are both $B_k$-symmetric and $T_k$-symmetric for every $1 \leq k < n$. So $u \cdot x \in B_n \subseteq B_{n-1}$ implies $u \cdot y \in B_{n-1}$ and clause (v) of Lemma \ref{LEM BIJECT} implies that there is $h \in G$ with $\theta_{n-1}(u \cdot x) = h u \cdot x$ and $\theta_{n-1}(u \cdot y) = h u \cdot y$ (we use here the fact that $x, y \in X'$ implies $u \cdot x, u \cdot y \not\in E_{n-1}$ and thus $\theta_{n-1}(u \cdot x)$ and $\theta_{n-1}(u \cdot y)$ are defined). Then the letter $4$ appears once in $R(hu \cdot x) = R(\theta_{n-1}(u \cdot x))$ (since $u \cdot x \in B_{n-1} \cap B_n$) but appears twice in $R(h u \cdot y) = R(\theta_{n-1}(u \cdot y))$ (since $u \cdot y \in B_{n-1} \setminus B_n$). Thus $R(hu \cdot x) \neq R(hu \cdot y)$.

\underline{Case 2}: For every $n \geq 1$ $x$ and $y$ are $B_n$-symmetric. Fix $u \in G$ with $\alpha(u \cdot x) \neq \alpha(u \cdot y)$ (such a $u$ exists since $x, y \in X' \subseteq X''$). If $u \cdot x$ is not in $B_1$ then neither is $u \cdot y$, and we have that $L(u \cdot x)$ and $L(u \cdot y)$ are prefixes of $R(u \cdot x)$ and $R(u \cdot y)$, respectively. Thus $R(u \cdot x) \neq R(u \cdot y)$ if $u \cdot x \not\in B_1$. So suppose that $u \cdot x \in B_1$. We have $|L(u \cdot x)| = C + n$, where $n$ is maximal with $u \cdot x \in B_n$. Since $x$ and $y$ are $B_k$-symmetric for every $k \geq 1$, we must have $|L(u \cdot x)| = |L(u \cdot y)|$. Since $L(u \cdot x) \neq L(u \cdot y)$ and $|L(u \cdot x)| = |L(u \cdot y)|$, there is $k \geq 1$ with $L(u \cdot x)(k) \neq L(u \cdot y)(k)$. If $k \leq C$ then from the first case in the definition of $R$ it follows that $R(u \cdot x)(k) \neq R(u \cdot y)(k)$ and thus $R(u \cdot x) \neq R(u \cdot y)$. If $k > C$ then $u \cdot x, u \cdot y \in B_{k - C}$. Our symmetry assumption and clause (v) of Lemma \ref{LEM BIJECT} imply that there is $h \in G$ with $\theta_{k-C}(u \cdot x) = hu \cdot x$ and $\theta_{k-C}(u \cdot y) = hu \cdot y$ (as in Case 1, we again use the fact that $u \cdot x, u \cdot y \not\in E_{k-C}$ since $x, y \in X'$). Then $L(u \cdot x)(k)$ is the letter in $R(h u \cdot x)$ which follows the first occurrence of $4$, and $L(u \cdot y)(k)$ is the letter in $R(h u \cdot y)$ which follows the first occurrence of $4$. Therefore $R(hu \cdot x) \neq R(hu \cdot y)$.
\end{proof}

\section{Counter-examples for f-invariant and sofic entropies} \label{SEC COUNTER}

In this section we prove Propositions \ref{INTRO PROPF} and \ref{INTRO PROPS}. We handle f-invariant entropy first.

We remind the reader the definition of an induced action. Let $G$ be a countable group and let $H \leq G$ be a subgroup of finite index. Let $G / H$ denote the set of left $H$-cosets $\{g H \: g \in G\}$, and let $\zeta$ be the uniform probability measure on $G / H$. We let $G$ act on $(G / H, \zeta)$ by defining $g \cdot (a H) = g a H$. Fix any function $\sigma: G / H \rightarrow G$ with $\sigma(H) = 1_G$ and $\sigma(g H) \in g H$ for all $g \in G$. We abuse notation and let $\sigma(g)$ denote $\sigma(g H)$ for $g \in G$. Let $\gamma: (G / H) \times G \rightarrow H$ be the cocycle defined by
$$\gamma(a H, g) = \sigma(g a)^{-1} \cdot g \cdot \sigma(a).$$
If $H$ acts by measure preserving bijections on a probability space $(Y, \nu)$, then we define a measure preserving action of $G$ on the probability space $((G / H) \times Y, \zeta \times \nu)$ by
$$g \cdot (a H, y) = (g a H, \gamma(a H, g) \cdot y).$$
One can check that this is a well defined action of $G$. It is called the action of $G$ \emph{induced} from $H \acts (Y, \nu)$. It is well known that the induced action of $G$ is ergodic if and only if $H \acts (Y, \nu)$ is ergodic \cite{Z84}.

\begin{prop}
Let $G$ be a finitely generated non-cyclic free group. For every $h \in \R$,
$$\sup_{G \acts (X, \mu)} \Delta^*_G(X, \mu) = \sup_{G \acts (X, \mu)} \Delta_G(X, \mu) = + \infty,$$
where the supremums are taken over all essentially free ergodic probability measure preserving actions $G \acts (X, \mu)$ with $f_G(X, \mu)$ defined and $f_G(X, \mu) = h$.
\end{prop}

\begin{proof}
We will use the notations and definitions of \cite{S12b}. Fix a finitely generated non-cyclic free group $G$, fix $h \in \R$, and let $M > 0$. We will construct an essentially free ergodic action $G \acts (X, \mu)$ such that $f_G(X, \mu)$ is defined, $f_G(X, \mu) = h$, and $\log(\Delta_G(X, \mu)) \geq \Delta^*_G(X, \mu) > M$.

Let $r > 1$ be the rank of $G$. Fix $n > \exp(\frac{M - h}{r-1})$. Let $\nu$ be a probability measure on $\N$ satisfying
$$\sH(\nu) = \sum_{k \in \N} -\nu(k) \cdot \log(\nu(k)) =  n \cdot h + n (r - 1) \cdot \log(n).$$
Notice that the right hand side is positive since $n > \exp(\frac{-h}{r-1})$ and thus such a probability measure $\nu$ exists. Let $K$ be a normal subgroup of $G$ with $|G : K| = n$. Consider the Bernoulli shift $K \acts (\N^K, \nu^K)$. By \cite{B10a} we have that $f_K(\N^K, \nu^K)$ is defined and
$$f_K(\N^K, \nu^K) = \sH(\nu) = n \cdot h + n (r - 1) \cdot \log(n).$$
Let $\zeta$ be the uniform probability measure on $G / K$, set $(X, \mu) = ((G / K) \times \N^K, \zeta \times \nu^K)$, and let $G \acts (X, \mu)$ be the action of $G$ induced from $K \acts (\N^K, \nu^K)$. Since $K \acts (\N^K, \nu^K)$ is ergodic, $G \acts (X, \mu)$ is ergodic as well. It is easy to see that $K \acts (X, \mu)$ has $n$ ergodic components, namely $\{g K\} \times \N^K$ for $g K \in G / K$, and the action of $K$ on any of its ergodic components is measurably conjugate to $K \acts (\N^K, \nu^K)$. The sets $\{(a K, y) \in (G / K) \times \N^K \: a K = g K, \ y(1_K) = t\}$ for $g K \in G / K$ and $t \in \N$ form a generating partition for $K \acts (X, \mu)$, and it is readily checked that this partition has Shannon entropy $\sH(\nu) + \log(n) < \infty$. Therefore $f_K(X, \mu)$ is defined. The rank of $K$, $r_K$, is related to its index, $n$, by the formula $r_K = n \cdot (r - 1) + 1$ [Proposition I.3.9 \cite{LS77}]. So by the ergodic decomposition formula \cite{S12b}
$$f_K(X, \mu) = f_K(\N^K, \nu^K) - (r_K - 1) \log(n) = f_K(\N^K, \nu^K) - n (r - 1) \log(n) = n \cdot h.$$
Now by the subgroup formula \cite{S12a} we have
$$f_G(X, \mu) = \frac{1}{n} \cdot f_K(X, \mu) = h.$$
Since $\mu$ has no atoms, it follows from the main theorem of \cite{S12b} that $G \acts (X, \mu)$ is essentially free.

Now suppose that $\alpha$ is a generating partition for $G \acts (X, \mu)$ with $\sH(\alpha) < \infty$. Fix $g_1, g_2, \ldots, g_n \in G$ with $g_1 = 1_G$ and $G / K = \{g_i K \: 1 \leq i \leq n\}$. Enumerate the $K$-ergodic measures in the support of $\mu$ as $\mu_1, \mu_2, \ldots, \mu_n$ so that $\mu_i(\{g_i K\} \times \N^K) = 1$ for each $1 \leq i \leq n$. Notice that $\mu_1 = g_i^{-1} \cdot \mu_i$. By \cite[Lemma 4.2 (ii)]{S12b} we have
$$\frac{1}{n} \cdot \sum_{i = 1}^n \sH_{\mu_i}(\alpha) \leq \sH_\mu(\alpha).$$
Let $\beta$ be the restriction of $\bigvee_{i = 1}^n g_i^{-1} \cdot \alpha$ to $\{K\} \times \N^K$. Then
$$\sH_{\mu_1}(\beta) \leq \sum_{i = 1}^n \sH_{\mu_1}(g_i^{-1} \cdot \alpha) = \sum_{i = 1}^n \sH_{\mu_i}(\alpha) \leq n \cdot \sH(\alpha).$$
Since $K$ is normal we have
$$G \cdot \alpha = K \cdot \{g_1, g_2, \ldots, g_n\} \cdot \alpha$$
and hence $\beta$ is a generating partition for $K \acts (\{K\} \times \N^K, \mu_1)$. It follows that $f_K(\{K\} \times \N^K, \mu_1) \leq \sH_{\mu_1}(\beta)$ \cite{B10a}. However, $K \acts (\{K\} \times \N^K, \mu_1)$ is measurably conjugate to $K \acts (\N^K, \nu^K)$ and therefore
$$n \cdot h + n (r - 1) \cdot \log(n) = f_K(\N^K, \nu^K) = f_K(\{K\} \times \N^K, \mu_1) \leq \sH_{\mu_1}(\beta) \leq n \cdot \sH(\alpha).$$
Since $\alpha$ was an arbitrary finite Shannon entropy generating partition for $G \acts (X, \mu)$ it follows that
$$\Delta^*_G(X, \mu) \geq h + (r - 1) \cdot \log(n) > M.$$
Finally, we of course always have the inequality $\log(\Delta_G(X, \mu)) \geq \Delta^*_G(X, \mu)$.
\end{proof}

We point out that the action $G \acts (X, \mu)$ constructed in the above proof factors onto $G \acts (G / K, \zeta)$ which has f-invariant entropy $f_G(G / K, \zeta) = - (r - 1) \cdot \log(n) < 0$ (\cite[Lemma 2.3]{S12b}). We do not know if the above proposition is still true if in addition to picking $h \in \R$ one picks a constant $c > 0$ and considers essentially free ergodic actions $G \acts (X, \mu)$ which not only satisfy $f_G(X, \mu) = h$ but also satisfy $f_G(Y, \nu) \geq - c$ for every factor $(Y, \nu)$ of $(X, \mu)$. This additional requirement could potentially lead to a result similar to Krieger's finite generator theorem.

Now we consider the case of sofic entropy.

\begin{lem}
There is a sofic group $G$, a sofic approximation sequence $\Sigma$ to $G$, and a normal subgroup $K \lhd G$ of finite index such that $h^\Sigma_G(G / K, \zeta) = - \infty$, where $\zeta$ is the uniform probability measure on $G / K$.
\end{lem}

\begin{proof}
For the sake of brevity we give a simple example of such a group. However the situation described in the lemma should occur whenever $G$ has a subgroup of finite index and admits a sofic approximation sequence coming from a sequence of expander graphs.

We will use some of the notation and definitions from \cite{B10c}. Let $G$ be a finitely generated non-cyclic free group. It is well known that such groups are sofic. Fix a normal subgroup $K$ of $G$ of index $2$. Let $\zeta$ be the uniform probability measure on $G / K$. By \cite[Lemma 2.3]{S12b}, the f-invariant entropy of this action is
$$f_G(G / K, \zeta) = -(r - 1) \cdot \log(2) < 0,$$
where $r > 1$ is the rank of $G$. Let $\phi : G / K \rightarrow \{0, 1\}$ be a bijection. In \cite{B10c}, Bowen proved that f-invariant entropy can be obtained by considering random homomorphisms into finite symmetric groups and then performing computations similar to those used in defining sofic entropy. In particular, his theorem implies that we can find a sofic approximation sequence $\Sigma = \{\sigma_i \: i \in \N\}$ to $G$ consisting of homomorphisms $\sigma_i : G \rightarrow \Sym(m_i)$ such that
$$|\psi: \{1, 2, \ldots, m_i\} \rightarrow \{0, 1\} \: d^H_{\sigma_i} (\phi, \psi) \leq \epsilon\}| = 0$$
for sufficiently small $\epsilon > 0$, sufficiently large finite sets $H \subseteq G$, and sufficiently large $i \in \N$. Using the definition of sofic entropy given in \cite{B10b}, it immediately follows from the previous sentence that $h^\Sigma_G(G / K, \zeta) = -\infty$.
\end{proof}

\begin{prop}
There exists a sofic group $G$, a sofic approximation sequence $\Sigma$ to $G$, and an essentially free ergodic action $G \acts (X, \mu)$ such that $h^\Sigma_G(X, \mu) = -\infty$ but $\Delta^*_G(X, \mu) = \Delta_G(X, \mu) = +\infty$.\end{prop}

\begin{proof}
Take any sofic group $G$ and sofic approximation sequence $\Sigma$ to $G$ with the properties that there is a normal subgroup $K \lhd G$ of finite index such that $h^\Sigma_G(G / K, \zeta) = -\infty$, where $\zeta$ is the uniform probability measure on $G / K$. Consider the Bernoulli shift $K \acts ([0, 1]^K, \lambda^K)$, where $\lambda$ is Lebesgue measure on the interval $[0, 1]$. Let $G \acts (X, \mu)$ be the induced action of $G$, where $X = (G / K) \times [0, 1]^K$ and $\mu = \zeta \times \lambda^K$. Since $K \acts ([0, 1]^K, \lambda^K)$ is ergodic, so is $G \acts (X, \mu)$. If $g \in G$ has a non-trivial power $g^n$ lying in $K$, then the set of points in $X$ fixed by $g$ must have measure $0$ since the action of $K$ is essentially free. On the other hand, if $g$ has no non-trivial power lying in $K$, then $g$ acts freely on $G / K$ (since $K$ is normal) and thus acts freely on $X$. So the action $G \acts (X, \mu)$ is essentially free. Clearly $G \acts (X, \mu)$ factors onto $G \acts (G / K, \zeta)$. Let $\xi$ be the finite partition of $X$ associated to this factor map. Using Kerr's definition of sofic entropy \cite{Ke}, we can work with partitions finer than $\xi$ and use the fact that $h^\Sigma_G(G / K, \zeta) = -\infty$ to quickly obtain $h^\Sigma_G(X, \mu) = -\infty$. As was shown in the proof of the previous corollary, any finite Shannon entropy generator for $G \acts (X, \mu)$ would provide a finite Shannon entropy generator for $K \acts ([0, 1]^K, \lambda^K)$ (see the role of $\alpha$ and $\beta$ in that proof). However, Kerr and Li \cite{KL11b} proved that $K \acts ([0, 1]^K, \lambda^K)$ does not admit any generating partition having finite Shannon entropy (we use here the fact that a subgroup of a sofic group is sofic \cite{P08}). Thus $\Delta^*_G(X, \mu) = \Delta_G(X, \mu) = + \infty$.
\end{proof}

It is unknown to the author if Krieger's theorem and Rohlin's theorem hold for actions of sofic groups for which the sofic entropy is not negative infinity. We mention that the potential dependence of sofic entropy on the choice of a sofic approximation sequence clearly poses a potential obstruction to Rohlin's theorem.

\end{document}